\documentclass[envcountsame,envcountresetchap,leqno]{svmult}

\usepackage{makeidx}
\usepackage{graphicx}
\usepackage{multicol}
\usepackage[bottom]{footmisc}
\usepackage{amsmath,amssymb,amsfonts,amscd,amsopn}
\usepackage{hyperref}
\usepackage{eucal, mathrsfs}
\usepackage[all]{xy}\xyoption{dvips}
\usepackage{enumitem}
\RequirePackage[T1]{fontenc}
\smartqed
\setitemindent{(3.3.3)}
\DeclareMathOperator{\conv}{{Conv}}

\begin{document}
\setcounter{page}{1}

%%%%%%%%%%%%%%%%%%%%%%%%%%%%%%%%%%%%%%
\title*{The cone of curves of K3 surfaces revisited}%
\author{S\'andor J Kov\'acs}%
\institute{University of Washington\\ Department of Mathematics\\ Box
  354350\\ Seattle, WA 98195-4350, USA\\
\texttt{skovacs@uw.edu}\\
\texttt{http://www.math.washington.edu/$\sim$kovacs}\\ \ }
%%%%%%%%%%%%%%%%%%%%%%%%%%%%%%%%%%%%%%

% \abstract{The description of the cone of curves of K3 surfaces is extended to
%   arbitrary characteristics.  \keywords{cone of curves, K3 surface}
%   \subjclass[2000]{14}}
\maketitle

%%%%%%%%%%%%%%%%%%%%%%%%%%%%%%%%%%%%%%%%%%%%%%%%%%
%\theoremstyle{plain}
\spnewtheorem{proclaim}{Theorem}[section]{\bf}{\it}
\spnewtheorem{thm}[proclaim]{Theorem}{\bf}{\it}
\spnewtheorem{mainthm}[proclaim]{Main Theorem}{\bf}{\it}
\spnewtheorem{cor}[proclaim]{Corollary} {\bf}{\it}
\spnewtheorem{cors}[proclaim]{Corollaries} {\bf}{\it}
\spnewtheorem{lem}[proclaim]{Lemma} {\bf}{\it}
\spnewtheorem{prop}[proclaim]{Proposition} {\bf}{\it}
\spnewtheorem{conj}[proclaim]{Conjecture}{\bf}{\it}
\spnewtheorem{subproclaim}[equation]{Theorem}{\bf}{\it}
\spnewtheorem{subthm}[equation]{Theorem}{\bf}{\it}
\spnewtheorem{subcor}[equation]{Corollary} {\bf}{\it}
\spnewtheorem{sublem}[equation]{Lemma} {\bf}{\it}
\spnewtheorem{subprop}[equation]{Proposition} {\bf}{\it}
\spnewtheorem{subconj}[equation]{Conjecture}{\bf}{\it}
%%%%
%%%%%%%%%%%%%%%%%%%%%%%%%%%%%%
%\theoremstyle{remark}
\spnewtheorem{rem}[proclaim]{Remark}{\sc}{\rm}
\spnewtheorem{subrem}[equation]{Remark}{\sc}{\rm}
\spnewtheorem{notation}[proclaim]{Notation} {\sc}{\rm}
\spnewtheorem{assume}[proclaim]{Assumptions} {\sc}{\rm}
\spnewtheorem{obs}[proclaim]{Observation} {\sc}{\rm}
\spnewtheorem{subclaim}[equation]{Claim} {\sc}{\rm}
\spnewtheorem{defini}[proclaim]{Definition}{\sc}{\rm}
%%%%%%%%%%%%%%%%%%%%%%%%%%%%%%
\spnewtheorem{exmp}[proclaim]{Example}{\sc}{\rm}
\spnewtheorem{defno}[proclaim]{Definitions and Notation}{\sc}{\rm}
\spnewtheorem{defn}[proclaim]{Definition}{\sc}{\rm}
\spnewtheorem{subdefn}[equation]{Definition}{\sc}{\rm}
%%%%%%%%%%%%%%%%%%%%%%%%%%%%%%%%%%%%%%%%%%%%%%%%%%
\numberwithin{equation}{proclaim}
%%%%%%%%%%%%%%%%%%%%%%%%%%%%%%%%%%%%%%

%%%%%%%%%%%%%%%%%%%%%%%%%%%%%%%%%%%%%%
\newcommand{\rp}{\mathbb R_+}
\newcommand{\NE}{\ensuremath{\overline{N\!E}}}
%script letters  small s then capital letter
\newcommand{\sA}{\mathscr{A}}
\newcommand{\sB}{\mathscr{B}}
\newcommand{\sC}{\mathscr{C}}
\newcommand{\sD}{\mathscr{D}}
\newcommand{\sE}{\mathscr{E}}
\newcommand{\sF}{\mathscr{F}}
\newcommand{\sG}{\mathscr{G}}
\newcommand{\sH}{\mathscr{H}}
\newcommand{\sI}{\mathscr{I}}
\newcommand{\sJ}{\mathscr{J}}
\newcommand{\sK}{\mathscr{K}}
\newcommand{\sL}{\mathscr{L}}
\newcommand{\sM}{\mathscr{M}}
\newcommand{\sN}{\mathscr{N}}
\newcommand{\sO}{\mathscr{O}}
\newcommand{\sP}{\mathscr{P}}
\newcommand{\sQ}{\mathscr{Q}}
\newcommand{\sR}{\mathscr{R}}
\newcommand{\sS}{\mathscr{S}}
\newcommand{\sT}{\mathscr{T}}
\newcommand{\sU}{\mathscr{U}}
\newcommand{\sV}{\mathscr{V}}
\newcommand{\sW}{\mathscr{W}}
\newcommand{\sX}{\mathscr{X}}
\newcommand{\sY}{\mathscr{Y}}
\newcommand{\sZ}{\mathscr{Z}}

% Sans Serif capitals
\newcommand{\sfA}{{\sf A}}
\newcommand{\sfB}{{\sf B}}
\newcommand{\sfC}{{\sf C}}
\newcommand{\sfD}{{\sf D}}
\newcommand{\sfE}{{\sf E}}
\newcommand{\sfF}{{\sf F}}
\newcommand{\sfG}{{\sf G}}
\newcommand{\sfH}{{\sf H}}
\newcommand{\sfI}{{\sf I}}
\newcommand{\sfJ}{{\sf J}}
\newcommand{\sfK}{{\sf K}}
\newcommand{\sfL}{{\sf L}}
\newcommand{\sfM}{{\sf M}}
\newcommand{\sfN}{{\sf N}}
\newcommand{\sfO}{{\sf O}}
\newcommand{\sfP}{{\sf P}}
\newcommand{\sfQ}{{\sf Q}}
\newcommand{\sfR}{{\sf R}}
\newcommand{\sfS}{{\sf S}}
\newcommand{\sfT}{{\sf T}}
\newcommand{\sfU}{{\sf U}}
\newcommand{\sfV}{{\sf V}}
\newcommand{\sfW}{{\sf W}}
\newcommand{\sfX}{{\sf X}}
\newcommand{\sfY}{{\sf Y}}
\newcommand{\sfZ}{{\sf Z}}

% Blackboard capitals
\newcommand{\bA}{\mathbb{A}}
\newcommand{\bB}{\mathbb{B}}
\newcommand{\bC}{\mathbb{C}}
\newcommand{\bD}{\mathbb{D}}
\newcommand{\bE}{\mathbb{E}}
\newcommand{\bF}{\mathbb{F}}
\newcommand{\bG}{\mathbb{G}}
\newcommand{\bH}{\mathbb{H}}
\newcommand{\bI}{\mathbb{I}}
\newcommand{\bJ}{\mathbb{J}}
\newcommand{\bK}{\mathbb{K}}
\newcommand{\bL}{\mathbb{L}}
\newcommand{\bM}{\mathbb{M}}
\newcommand{\bN}{\mathbb{N}}
\newcommand{\bO}{\mathbb{O}}
\newcommand{\bP}{\mathbb{P}}
\newcommand{\bQ}{\mathbb{Q}}
\newcommand{\bR}{\mathbb{R}}
\newcommand{\bS}{\mathbb{S}}
\newcommand{\bT}{\mathbb{T}}
\newcommand{\bU}{\mathbb{U}}
\newcommand{\bV}{\mathbb{V}}
\newcommand{\bW}{\mathbb{W}}
\newcommand{\bX}{\mathbb{X}}
\newcommand{\bY}{\mathbb{Y}}
\newcommand{\bZ}{\mathbb{Z}}

\newcommand{\cA}{\mathcal{A}}
\newcommand{\cB}{\mathcal{B}}
\newcommand{\cC}{\mathcal{C}}
\newcommand{\cD}{\mathcal{D}}
\newcommand{\cE}{\mathcal{E}}
\newcommand{\cF}{\mathcal{F}}
\newcommand{\cG}{\mathcal{G}}
\newcommand{\cH}{\mathcal{H}}
\newcommand{\cI}{\mathcal{I}}
\newcommand{\cJ}{\mathcal{J}}
\newcommand{\cK}{\mathcal{K}}
\newcommand{\cL}{\mathcal{L}}
\newcommand{\cM}{\mathcal{M}}
\newcommand{\cN}{\mathcal{N}}
\newcommand{\cO}{\mathcal{O}}
\newcommand{\cP}{\mathcal{P}}
\newcommand{\cQ}{\mathcal{Q}}
\newcommand{\cR}{\mathcal{R}}
\newcommand{\cS}{\mathcal{S}}
\newcommand{\cT}{\mathcal{T}}
\newcommand{\cU}{\mathcal{U}}
\newcommand{\cV}{\mathcal{V}}
\newcommand{\cW}{\mathcal{W}}
\newcommand{\cX}{\mathcal{X}}
\newcommand{\cY}{\mathcal{Y}}
\newcommand{\cZ}{\mathcal{Z}}

\newcommand{\frA}{\mathfrak{A}}
\newcommand{\frB}{\mathfrak{B}}
\newcommand{\frC}{\mathfrak{C}}
\newcommand{\frD}{\mathfrak{D}}
\newcommand{\frE}{\mathfrak{E}}
\newcommand{\frF}{\mathfrak{F}}
\newcommand{\frG}{\mathfrak{G}}
\newcommand{\frH}{\mathfrak{H}}
\newcommand{\frI}{\mathfrak{I}}
\newcommand{\frJ}{\mathfrak{J}}
\newcommand{\frK}{\mathfrak{K}}
\newcommand{\frL}{\mathfrak{L}}
\newcommand{\frM}{\mathfrak{M}}
\newcommand{\frN}{\mathfrak{N}}
\newcommand{\frO}{\mathfrak{O}}
\newcommand{\frP}{\mathfrak{P}}
\newcommand{\frQ}{\mathfrak{Q}}
\newcommand{\frR}{\mathfrak{R}}
\newcommand{\frS}{\mathfrak{S}}
\newcommand{\frT}{\mathfrak{T}}
\newcommand{\frU}{\mathfrak{U}}
\newcommand{\frV}{\mathfrak{V}}
\newcommand{\frW}{\mathfrak{W}}
\newcommand{\frX}{\mathfrak{X}}
\newcommand{\frY}{\mathfrak{Y}}
\newcommand{\frZ}{\mathfrak{Z}}

\newcommand{\fra}{\mathfrak{a}}
\newcommand{\frb}{\mathfrak{b}}
\newcommand{\frc}{\mathfrak{c}}
\newcommand{\frd}{\mathfrak{d}}
\newcommand{\fre}{\mathfrak{e}}
\newcommand{\frf}{\mathfrak{f}}
\newcommand{\frg}{\mathfrak{g}}
\newcommand{\frh}{\mathfrak{h}}
\newcommand{\fri}{\mathfrak{i}}
\newcommand{\frj}{\mathfrak{j}}
\newcommand{\frk}{\mathfrak{k}}
\newcommand{\frl}{\mathfrak{l}}
\def\frm{\mathfrak{m}}
\newcommand{\frn}{\mathfrak{n}}
\newcommand{\fro}{\mathfrak{o}}
\newcommand{\frp}{\mathfrak{p}}
\def\frq{\mathfrak{q}}
\newcommand{\frr}{\mathfrak{r}}
\newcommand{\frs}{\mathfrak{s}}
\newcommand{\frt}{\mathfrak{t}}
\newcommand{\fru}{\mathfrak{u}}
\newcommand{\frv}{\mathfrak{v}}
\newcommand{\frw}{\mathfrak{w}}
\newcommand{\frx}{\mathfrak{x}}
\newcommand{\fry}{\mathfrak{y}}

\newcommand{\fA}{\mathfrak{A}}
\newcommand{\fB}{\mathfrak{B}}
\newcommand{\fC}{\mathfrak{C}}
\newcommand{\fD}{\mathfrak{D}}
\newcommand{\fE}{\mathfrak{E}}
\newcommand{\fF}{\mathfrak{F}}
\newcommand{\fG}{\mathfrak{G}}
\newcommand{\fH}{\mathfrak{H}}
\newcommand{\fI}{\mathfrak{I}}
\newcommand{\fJ}{\mathfrak{J}}
\newcommand{\fK}{\mathfrak{K}}
\newcommand{\fL}{\mathfrak{L}}
\newcommand{\fM}{\mathfrak{M}}
\newcommand{\fN}{\mathfrak{N}}
\newcommand{\fO}{\mathfrak{O}}
\newcommand{\fP}{\mathfrak{P}}
\newcommand{\fQ}{\mathfrak{Q}}
\newcommand{\fR}{\mathfrak{R}}
\newcommand{\fS}{\mathfrak{S}}
\newcommand{\fT}{\mathfrak{T}}
\newcommand{\fU}{\mathfrak{U}}
\newcommand{\fV}{\mathfrak{V}}
\newcommand{\fW}{\mathfrak{W}}
\newcommand{\fX}{\mathfrak{X}}
\newcommand{\fY}{\mathfrak{Y}}

\newcommand{\ff}{\mathfrak{f}}
\newcommand{\fg}{\mathfrak{g}}
\newcommand{\fm}{\mathfrak{m}}
\newcommand{\fu}{\mathfrak{u}}

\newcommand{\al}{\alpha}
\newcommand{\be}{\beta}
\newcommand{\ga}{\gamma}
\newcommand{\de}{\delta}
\newcommand{\pa}{\partial}   %pretend its Greek
\newcommand{\delbar}{\overline{\partial}}
\newcommand{\epz}{\varepsilon}
\newcommand{\ph}{\phi}
\newcommand{\phz}{\varphi}
\newcommand{\et}{\eta}
\newcommand{\io}{\iota}
\newcommand{\ka}{\kappa}
\newcommand{\la}{\lambda}
\newcommand{\tha}{\theta}
\newcommand{\thz}{\vartheta}
\newcommand{\rh}{\rho}
\newcommand{\si}{\sigma}
\newcommand{\ta}{\tau}
\newcommand{\ch}{\chi}
\newcommand{\ps}{\psi}
\newcommand{\ze}{\zeta}
\newcommand{\om}{\omega}
\newcommand{\GA}{\Gamma}
\newcommand{\LA}{\Lambda}
\newcommand{\DE}{\Delta}
\newcommand{\SI}{\Sigma}
\newcommand{\THA}{\Theta}
\newcommand{\OM}{\Omega}
\newcommand{\XI}{\Xi}
\newcommand{\UP}{\Upsilon}
\newcommand{\PI}{\Pi}
\newcommand{\PS}{\Psi}
\newcommand{\PH}{\Phi}

\newcommand{\com}{\circ}     % composition of functions
\newcommand{\iso}{\simeq}    % preferred isomorphism symbol
\newcommand{\ten}{\varotimes}   % tensor product
\newcommand{\add}{\oplus}    % direct sum
\newcommand{\lbar}{\overline}
\newcommand{\ol}{\overline}
\newcommand{\vvert}{\,\vert\,}

\newcommand{\ul}{\underline}
\newcommand{\lineq}{\sim}
\newcommand{\simq}[0]{\sim_{\bQ}}
\newcommand{\tcdot}{\cdot}

\newcommand{\rtarr}{\longrightarrow}
\newcommand{\ltarr}{\longleftarrow}
\newcommand{\from}{\longleftarrow}
\newcommand{\monoto}{\lhook\joinrel\relbar\joinrel\rightarrow}
\newcommand{\epito}{\relbar\joinrel\twoheadrightarrow}
\newcommand{\into}{\hookrightarrow}
\newcommand{\longinto}{\lhook\joinrel\relbar\joinrel\longrightarrow}
\newcommand{\onto}{\twoheadrightarrow}
\newcommand{\isom}{\overset{\simeq\ }\longrightarrow}
\newcommand{\ideal}{\vartriangleleft}
\newcommand{\wigglyarrow}{\rightsquigarrow}
\newcommand{\wtilde}{\widetilde}
\newcommand{\wt}{\widetilde}
\newcommand{\what}{\widehat}
\newcommand{\od}[1]{\ensuremath{\mcd(\!{#1}\!)}}
\newcommand{\rup}[1]{\lceil{#1}\rceil}
\newcommand{\rdown}[1]{\lfloor{#1}\rfloor}
\newcommand{\rupplus}[1]{\left\lceil{#1}\right\rceil}
\newcommand{\rdownplus}[1]{\left\lfloor{#1}\right\rfloor}

%%%%%%%%% math short hand
\newcommand{\Cal}{\mathscr}
\newcommand{\leteq}{\colon\!\!\!=}
\newcommand{\col}{\colon}
%%%% gothic
\newcommand{\Mg}{\mathfrak M_g}
\newcommand{\Mgbar}{\overline{\mathfrak M}_g}
\newcommand{\Mq}{\mathfrak M_q}
\newcommand{\Mqbar}{\overline{\mathfrak M}_q}
\newcommand{\Mh}{\mathfrak M_h}
\newcommand{\Mhbar}{\overline{\mathfrak M}_h}
\newcommand{\maxid}{\mathfrak m}
%%%% curly
\newcommand{\MP}{\mcM\!\mcP}
%%%% Blackboard bold
%%%% sans serif
\newcommand{\Schs}{\ensuremath{{\sf{Sch}_{S}}}\xspace}
\newcommand{\LocFrees}{\ensuremath{{\sf{LocFree}_{S}}}\xspace}
\newcommand{\Ab}{\ensuremath{{\sf{Ab}}}\xspace}
\newcommand{\Mod}{\ensuremath{{\sf{Mod}}}\xspace}
\newcommand{\sets}{\ensuremath{{\sf{Sets}}}\xspace}
\newcommand{\sch}{\ensuremath{{\sf{Sch}}}\xspace}
\newcommand{\hilb}{\ensuremath{{\mcH\mci\mcl\mcb}}\xspace}
\newcommand{\grass}{\ensuremath{{\mcG\mcr\mca\mcs\mcs}}\xspace}
\newcommand{\kSch}{\ensuremath{\sf{Sch}}\xspace}
\newcommand{\kCat}{\ensuremath{\sf{Cat}}\xspace}
\newcommand{\kHrc}{\ensuremath{\sf{Hrc}}\xspace}
\newcommand{\uSch}{\underline{\kSch}\xspace}
\newcommand{\uCat}{\underline{\kCat}\xspace}
\newcommand{\uTwo}{\underline{2}}
\newcommand{\uOne}{\underline{1}}
\newcommand{\uZero}{\underline{0}}
\newcommand{\uFaisc}{{\sf Sheaves}}
%%%%%%%%%%%%%%%%%%%%%%%
\newcommand{\infinity}{\infty}
%%%%
\newcommand{\sub}{\subseteq}
\newcommand{\tensor}{\varotimes}
\newcommand{\lpback}[1]{\mcL{#1}^*}
\newcommand{\hypush}{\myR\phi_{\cdot *}}
\newcommand{\rpforward}[1]{\myR{#1}_*}
\newcommand{\ratmap}{\dasharrow} 
\newcommand{\eq}{\equiv}
\newcommand{\myquad}{\ }
\newcommand{\normalsgp}{\,\triangleleft\,}
%%%
\newcommand{\longto}{\longrightarrow}
\newcommand{\factor}[2]{\left. \raise 2pt\hbox{\ensuremath{#1}} \right/
        \hskip -2pt\raise -2pt\hbox{\ensuremath{#2}}}
%%%%%%%%%%%%%%%%%%%%%%%%%%%%%%%%%%%%%%%%%%%%%%%%%%%%%%%%%%%%
\newcommand{\karldot}{{{\,\begin{picture}(1,1)(-1,-2)\circle*{2}\end{picture}\ }}}
\newcommand{\nsubset}{\not\subset}
\newcommand{\disjoint}{\overset{_\kdot}\cup}
\newcommand{\myR}{{\mcR\!}}
\newcommand{\myRi}[1]{{\mcR^{\!#1}\!}}
\newcommand{\myH}{{\bf \mathcal H}}
\newcommand{\myL}{{\mcL}}%script letters  small s then capital letter
\newcommand{\tld}{\widetilde }
\newcommand{\blank}{\underline{\hskip 10pt}}
\newcommand{\hypertensor}{{\uline \tensor}}
\newcommand{\uline}{\underline}
\newcommand{\bproj}{\mathbf{Proj}}
\newcommand{\bolddot}{\kdot} %{\raisebox{.15em}{\ensuremath\centerdot}}}
\newcommand{\bdot}{\kdot}
\newcommand{\boldldot}{{{\ensuremath\centerdot}}}
\newcommand{\kdot}{{{\,\begin{picture}(1,1)(-1,-2)\circle*{2}\end{picture}\ }}}
\newcommand{\mydot}{\kdot}
\newcommand{\complex}[1]{{#1}^{\raisebox{.15em}{\ensuremath\kdot}}}
\newcommand{\cx}{\sf}
\newcommand{\DuBois}[1]{{\underline \Omega {}^0_{#1}}}
\newcommand{\FullDuBois}[1]{{\underline \Omega {}^{\mydot}_{#1}}}
\newcommand{\Om}{\underline{\Omega}}
\newcommand{\Ox}[2]{\underline{\Omega}_X^{\,#1}(\log{#2})}
\newcommand{\Oxc}[2]{\underline{\Omega}_{X/B}^{\,#1}(\log{#2})}
\newcommand{\Oy}[2]{{\Omega}_X^{\,#1}(\log{#2})}
\newcommand{\Oyc}[2]{{\Omega}_{X/B}^{\,#1}(\log{#2})}
\newcommand{\resto}{\big\vert_}
\newcommand{\rresto}{\right\vert_}
%%%%%%%%%%%
%%%%%%%%% bibliography helpers
%%%%%%%%%%%
\newcommand\toappear{\rm (to appear)}
%%%%%%%%%%%
%%%%%%%%% text abbreviations
%%%%%%%%%%%
\newcommand\cf{cf.\ \cite}
\newcommand\mycf{cf.\ \mycite}
\newcommand\te{there exist}
\newcommand\st{such that}
%%%%%%%%%%%
%%%%%%%%%%%
%%%%%%%%%%%%%%%%%%%%%%%%%%%%%%
%%%%%%%%%%%%%%%%%%%%%%%%%%%%%%

%\input intro.tex
\section{Introduction}

The following theorem was proved in \cite{Kovacs94} over the complex numbers. It
turns out that the proof given there works with very small adjustments in arbitrary
characteristic.  The main difference is that while in the original article we worked
in a real homology group of the surface in question, here everything takes place in
the group of $\bR$-cycles modulo numerical equivalence. The arguments are essentially
the same.

As already noted the purpose of this note is to verify the above statement, that is,
to prove the following.

\begin{thm}\label{thm:intro}
  Let $X$ be a K3 surface of Picard number at least three over an algebraically
  closed field of arbitrary characteristic. Then one of the following mutually
  exclusive conditions are satisfied:
  \begin{enumerate}
  \item $X$ does not contain any curve of negative self-intersection
  \item $\displaystyle \NE(X)=\overline{\sum \rp \ell}$ where the sum runs over the
    classes of all smooth rational curves on $X$.
  \end{enumerate}
\end{thm}

\begin{rem}
  To cover all cases one would also need to consider K3 surfaces with Picard number
  less than $3$. The case of Picard number $1$ is trivial and the case of Picard
  number $2$ is handled in \eqref{cor:rank-two}.
\end{rem}

\begin{defno}
%\noindent{\sc Definitions and Notation}\quad
  Let $k$ be an algebraically closed field of arbitrary characteristic. Everything
  will be defined over $k$. 

  A K3 surface is a smooth projective surface $X$ such that  $\omega_X\simeq \sO_X$
  and $h^1(X,\sO_X)=0$. 

  Let $Z_1(X)$ denote the free abelian group generated by the irreducible and reduced
  $1$-dimensional subvarieties of $X$. Elements of this group are called
  \emph{$1$-cycles} on $X$.  Two $1$-cycles $C_1,C_2\in Z_1(X)$ are called
  \emph{numerically equivalent} if for any Cartier divisor $D$ on $X$, the
  intersection numbers $D\cdot C_1$ and $D\cdot C_2$ agree. This relationship is
  denoted by $C_1\equiv C_2$. All $1$-cycles numerically equivalent to the $0$ cycle
  form a subgroup of $Z_1(X)$, and the quotient is denoted by $N_1(X)_{\bZ}$. By
  extension of scalars we define 
  $$
  N_1(X):= N_1(X)_{\bZ}\otimes_{\bZ}\bR.
  $$

  The effective $1$-cycles in $N_1(X)$ generate a subsemigroup denoted by
  $NE(X)\subseteq N_1(X)$. This subsemigroup is called the \emph{cone of effective
    curves}. The \emph{closed cone of effective curves} is the closure of this in
  $N_1(X)$:
  $$
  \NE(X):=\overline{NE(X)}\subseteq N_1(X).
  $$ 
  For more details about the construction and basic properties of this, one should
  consult \cite[II.4]{Kollar96}.

  From now on $X$ is assumed to be a smooth projective surface defined over $k$.  A
  class $\xi\in N_1(X)$ is called \emph{integral} if it can be represented by a
  divisor on $X$. It is called \emph{effective} (respectively \emph{ample}) if it is
  integral and can be represented by an effective (respectively ample) divisor.  A
  class $\xi\in \NE(X)$ is called \emph{extremal} if it cannot be written as the sum
  of two incomparable classes in $\NE(X)$.  The class of a smooth rational curve is
  called a \emph{nodal} class. The set of all nodal classes is denoted by $\sN(X)$.

  Let $h$ be an ample class and define $$\sQ(X):= \{ \xi\in N_1(X) \mid \xi\cdot h>0,
  \xi\cdot\xi=0\}$$ Note that by the positivity condition this is just half of a
  quadric cone.

  The convex hull of a set will be denoted by $\conv$.

  An open subset of the boundary of the cone is called \emph{circular} if the cone is
  not locally finitely generated at any point in the open set.
\end{defno}

%\bigskip

\noindent{\sc Acknowledgements}\\ \noindent
  This paper owes its existence to Max Lieblich who suggested that perhaps the
  results of \cite{Kovacs94} also hold in arbitrary characteristic.

  The author was supported in part by NSF Grant DMS-0856185, and the Craig McKibben
  and Sarah Merner Endowed Professorship in Mathematics at the University of
  Washington.

\section{Simple facts}

Let us start with an easy, well-known consequence of the Riemann-Roch theorem:

\begin{lem}
  \label{lem:riemann-roch}
  Let $a\in N_1(X)$ be an integral class. Then $a\cdot a$ is an even integer.  If
  furthermore $a\cdot a\geq -2$, then either $a$ or $-a$ is effective.
\end{lem}

\begin{proof}
  As $a$ is integral, there exists a line bundle $\sL$ representing $a$.  By
  Riemann-Roch
  \begin{equation}
    \label{eq:3}
    h^0(X,\sL)-h^1(X,\sL)+h^0(X,\sL^\vee)= \frac 12\, a\cdot a + 2
  \end{equation}
  and hence $a\cdot a$ has to be even.

  If $a\cdot a\geq -2$, then the right hand side of (\ref{eq:3}) is positive, so
  either $h^0(X,\sL)>0$ or $h^0(X,\sL^\vee)>0$, so either $a$ or $-a$ is effective.
  \qed
\end{proof}

\begin{cor}
  $\sQ(X)\subset\NE(X)$ and for any $e\in\sQ(X)$ the hyperplane
  $(e\cdot\blank=0)=\{\xi\in N_1(X)\mid e\cdot \xi=0\}$ is a supporting hyperplane of
  $\conv(\sQ(X))$.
\end{cor}

\begin{proof}
  The first statement follows directly from \eqref{lem:riemann-roch} and since for
  any $\xi\in \conv(\sQ(X))$ (effective) irreducible class, $\xi\cdot e\geq 0$, this
  also implies the second statement.
  \qed
\end{proof}

\begin{cor}\label{cor:ed-is-positive}
  Let $e,d\in \NE(X)$ such that $e\cdot e=0$ and $d\cdot d>0$. Then $d$ is in the
  interior of $\NE(X)$ and $e\cdot d>0$.
\end{cor}

\begin{proof}
  Since $d$ is in the interior of $\sQ(X)$ and $\sQ(X)\subset \NE(X)$, it follows
  that $d$ is in the interior of $\NE(X)$. Then since $(e\cdot\blank=0)$ is a
  supporting hyperplane, $d$ cannot be contained in it and hence $e\cdot d>0$.
  \qed
\end{proof}

Next we establish that nodal rays can only accumulate along the cone generated by
$\sQ(X)$.

\begin{lem}\label{lem:limits-of-nodal-rays}
  Let $\{\ell_n\}\subset \sN(X)$ be an infinite sequence of nodal classes such that
  $\xymatrix@+1em{\rp{\ell_n} \ar[r]_-{n\to+\infty} & \rp{\xi}}$ for some
  $\xi\in\NE(X)$. Then $\xi\cdot\xi=0$.
\end{lem}

\begin{proof}
  Let $\ell\in\sN(X)$. Then for infinitely many $n\in\bN$, $\ell_n\neq \ell$, so
  $\ell_n\cdot\ell\geq 0$ and hence $\xi\cdot\ell\geq 0$. Applying this with
  $\ell=\ell_n$ yields that $\xi\cdot \xi\geq 0$. On the other hand since
  $\ell_n\cdot\ell_n<0$, it follows that $\xi\cdot \xi\leq 0$ as well and so we must
  have $\xi\cdot\xi=0$.
  \qed
\end{proof}

%\begin{lem}
%  $\overline{\sum_{\ell\in\sN(X)}\rp{\ell}}\subset \sum_{\ell\in\sN(X)}\rp{\ell} +
%  \conv(\sQ(X))$.
%\end{lem}

%\begin{proof}
%  Let $A$ denote the cone generated by the elements of $\sN(X)$. Then by
%  \eqref{lem:limits-of-nodal-rays}, $\overline{A}\subset
%  \sum_{\ell\in\sN(X)}\rp{\ell} + \conv(\sQ(X))$ and hence the desired statement
%  follows by \cite[2.1.4]{Kovacs94}.
%\end{proof}

\begin{cor}\label{cor:accumulation-of-nodal-rays}
  Let $h\in\NE(X)$ be an ample class and $\varepsilon>0$ a real number. Define
  $$
  \sQ_\varepsilon(X):= \{ \xi\in N_1(X) \mid \xi\cdot h=1, \xi\cdot\xi\geq
  -\varepsilon\}.
  $$
  Then the number of nodal classes not contained in $\conv(\sQ_\varepsilon(X))$ is
  finite.
\end{cor}

\begin{proof}
  The set $\{ \xi\in \NE(X) \mid \xi\cdot h=1\}$ is compact by Kleiman's criterion
  \cite[1.18]{KM98} and hence any infinite set contained in it has an accumulation
  point.  By \eqref{lem:limits-of-nodal-rays} all accumulation points have to be
  contained in $\sQ_\varepsilon(X)$, which implies the desired statement.
  \qed
\end{proof}

\begin{cor}\label{cor:non-nodal-extremal}
  Let $\xi\in \NE(X)$ be an extremal class which is not a multiple of a nodal class.
  Then $\xi\cdot\xi=0$.
\end{cor}

\begin{proof}
  It follows from Riemann-Roch and \cite[II.4.14]{Kollar96} that
  \begin{equation}
    \label{eq:1}
    \NE(X)= \sum_{\ell\in\sN(X)}\rp{\ell} + \conv(\sQ(X)),
  \end{equation}
  in particular $\xi\cdot\xi\leq 0$, and then for any $\varepsilon>0$,
  \begin{equation}
    \label{eq:2}
    \NE(X)= \sum_{ \scriptsize\begin{matrix}\ell\in\sN(X)\\
        \ell\not\in\conv(\sQ_\varepsilon(X))\end{matrix}}\rp{\ell} +
    (\NE(X)\cap \conv(\sQ_\varepsilon(X))),
  \end{equation}
  where the above sum is finite by \eqref{cor:accumulation-of-nodal-rays}.

  Suppose that $\xi\cdot\xi<0$.  Without loss of generality we may assume that
  $\xi\cdot h=1$ and choose an $\varepsilon>0$ such that $\xi\not\in
  \conv(\sQ_\varepsilon(X))$. Then $\xi$ is a multiple of a nodal class by
  (\ref{eq:2}).
  \qed
\end{proof}

\section{Subcones generated by two elements}

The following is a simple, but important computation.

\begin{lem}\label{lem:rank-two-subcones}
  Let $e,d$ be effective classes such that %$e$ lies on the boundary of $\NE(X)$ and
  $e$ is indecomposable and $d\cdot d>0$. Let $L$ be the $2$-dimensional linear
  subspace generated by $e$ and $d$ in $N_1(X)$ and $\sC=L\cap\NE(X)$. Then
  \begin{enumerate}
  \item if $e\cdot e=0$, then there exists an $f\in \sC$ such that $f$ is effective,
    $e\cdot f>0$, $e$ and $f$ are on opposite sides of $d$, and $f\cdot f=0$, and  
  \item if $e\cdot e=-2$, then there exists an $f\in \sC$ such that $f$ is effective,
    $e\cdot f>0$, $e$ and $f$ are on opposite sides of $d$, and     $f \cdot f=0$ or
    $f\cdot f=-2$.
  \end{enumerate}
\end{lem}

\begin{proof}
  Let $A=d\cdot d$, $B=e\cdot d$, and $C=e\cdot e$. Set $f=\alpha d - \beta e$.
  Then
  $$ 
  f\cdot f= A\alpha^2 - 2B \alpha \beta + C\beta ^2.
  $$

  If $C=e\cdot e=0$, then $B=e\cdot d>0$ by \eqref{cor:ed-is-positive} and hence the
  equation
  $$
  0= A\alpha^2 - 2B \alpha\beta =\alpha (A\alpha - 2B\beta)
  $$
  has a positive integer solution, $\alpha= 2B$, $\beta=A$ such that the class
  $f=\alpha d-\beta e$ has $f\cdot f=0$.  Since $\alpha d= f + \beta e$, $f$ and $e$
  lie on opposite sides of $d$ and it also follows that $f$ is effective by
  \eqref{lem:riemann-roch} and since $e\cdot f>0$.

  If $C=e\cdot e=-2$, then set $x=B\alpha/2+\beta$, $y=\alpha/2$, and $N=2A+B^2$.
  Then
  \begin{multline*}
    f\cdot f= A\alpha^2 - 2B \alpha \beta -2 \beta ^2 =\\ = -2 \left(\left(
        B\alpha/2+\beta\right)^2 - (2A+B^2)(\alpha/2)^2\right)
    = -2(x^2-Ny^2).
  \end{multline*}

  Now if $N$ is a square, then as above there are two effective solutions for $f\cdot
  f=0$ and they are on opposite sides of $d$ and hence one of them is on the side of
  $d$ opposite to $e$.

  If $N$ is not a square then finding an $f$ with $f\cdot f=-2$ is equivalent to
  solving Pell's equation $x^2-Ny^2=1$ \cite[17.5.2]{MR1070716}. One may choose a
  solution with both $x,y>0$ which again ensures that $e$ and $f$ are on opposite
  sides of $d$ and that $e\cdot f>0$. This completes the proof.
  \qed
\end{proof}

\begin{cor}\label{cor:rank-two}
  If $\varrho(X)=2$, let $\NE(X)=\rp \xi +\rp \eta$. Then one of the following
  mutually exclusive cases hold:
  \begin{enumerate}
  \item Neither $\rp\xi$ nor $\rp\eta$ contain any effective classes.
  \item Both $\rp\xi$ and $\rp\eta$ contain an effective class of $0$ or $-2$
    self-intersection.
  \end{enumerate}
\end{cor}

\begin{proof}
  If $\sN(X)=\emptyset$, then the decomposition in (\ref{eq:1}) implies that
  $\NE(X)=\conv(\sQ(X))$. If there exists an integral (equivalently, effective) class
  in $\sQ(X)$, then by (\ref{lem:rank-two-subcones}.1) we are in case
  (\ref{cor:rank-two}.2). If there are no integral classes in $\sQ(X)$ then we are in
  case (\ref{cor:rank-two}.1).

  If $\sN(X)\neq\emptyset$, then by the decomposition in (\ref{eq:1}) and
  (\ref{lem:rank-two-subcones}.2) we are in case (\ref{cor:rank-two}.2).
  \qed
\end{proof}

\section{K3 surfaces containing a smooth rational curve}

\begin{thm}\label{thm:main}
  Let $X$ be a K3 surface and $\xi\in\NE(X)$ an extremal vector.  Assume that
  $\varrho(X)\geq 3$ and $X$ contains a smooth rational curve. Then
  \begin{enumerate}
  \item $\NE(X)$ has no circular part,
  \item There exists a sequence $\ell_n\in\sN(X)$ such that
      $\xymatrix@+1em{\rp{\ell_n} \ar[r]_-{n\to+\infty} & \rp{\xi}}$, and
  \item $$\NE(X)=\overline{\sum_{\ell\in\sN(X)}\rp{\ell}}.$$
  \end{enumerate}
\end{thm}

\begin{proof}
  Let $\ell$ denote the class of a smooth rational curve, guaranteed by the
  assumption and let $\sigma: X\to X'$ be the morphism contracting $\ell$ to a point.
  Recall that $X'$ is still projective and let $h_1$ and $h_2$ be two linearly
  independent ample classes on $X'$. Let $d_i=\sigma^*h_i$ for $i=1,2$ and observe
  that $d_1$ and $d_2$ are effective classes on $X$ such that $\ell\cdot
  d_1=\ell\cdot d_2=0$ and $\ell, d_1, d_2$ are linearly independent.

  Suppose there exists $U\subset \partial \NE(X)$ a non-empty open subset of
  $\partial \NE(X)$ such that $\rp U$ is a circular part (i.e., nowhere locally
  finitely generated) of $\NE(X)$. By \eqref{lem:limits-of-nodal-rays} it follows
  that $U\subseteq \sQ(X)$ and hence in a neighbourhood of $\rp U$ every effective
  class has non-negative self-intersection. 

  Let $h$ be an arbitrary ample class and observe that \eqref{lem:rank-two-subcones}
  implies that in the $2$-dimensional linear subspace generated by $h$ and $\ell$
  there is an effective class $f$ with either $f\cdot f=0$ or $f\cdot f=-2$ and such
  that $f$ is on the side of $h$ opposite to $\ell$.  We may repeat the same
  procedure with $\ell$ replaced by $f$ and $h$ replaced by another ample class and
  find that these classes are all over near the boundary of $\NE(X)$. In particular,
  we can find an ample class $h\in\NE(X)$ and an effective class $f$ with either
  $f\cdot f=0$ or $f\cdot f=-2$ such that the $2$-dimensional linear subspace
  generated by $h$ and $f$ intersects $U$ non-trivially. Then applying
  \eqref{lem:rank-two-subcones} again and combining it by the observation above we
  obtain that there exists an effective class $e\in U$ such that $e\cdot e=0$.

  Next let $d$ be one of $d_1$ and $d_2$ such that $e,d,\ell$ are linearly
  independent.  Let $A=d\cdot d$, $B=e\cdot d$, and $C=e\cdot \ell$. Recall that by
  the choice of $d$ we have $d\cdot\ell =0$.

  \setcounter{equation}3
  \begin{subclaim}\label{subclaim:nonzero}
    $2B^2\neq AC^2$
  \end{subclaim}
  \begin{proof}%[Proof of \eqref{subclaim:nonzero}]
    Suppose $2B^2= AC^2$ and let $f=ACe-BCd+B^2\ell$. Then $f\cdot f=
    B^2(AC^2-2B^2)=0$ and $f\cdot e=0$. Then by \eqref{lem:riemann-roch} $f$ or $-f$
    is effective. However, since $e\in U$ which is a circular part of
    $\partial\NE(X)$, the only effective classes contained in the hyperplane
    $(e\cdot\blank=0)$ are multiples of $e$. This implies that then $BC=B^2=0$, so
    $B=0$. Applying the same argument for $d$, it would follow that $d$ is a multiple
    of $e$ which is impossible by the choice of $d$ and $e$. Therefore we reached a
    contradiction and hence the claim is proven.  
    \qed
  \end{proof}

  \noindent
  {\it Continuation of the proof of \eqref{thm:main}.}  Next let $n\in\bN$ and for
  $n\geq 3$ define
  $$
  \pm d_n= (2(2B^2-AC^2)C n^2 - 4Bn) e+ (2C^2n) d + (1-2BCn) \ell.
  $$
  Then $d_n\cdot d_n=-2$ and by \eqref{lem:riemann-roch} either $d_n$ or $-d_n$ is
  effective. Choose $d_n$ to be effective. By \eqref{subclaim:nonzero} $2B^2-AC^2\neq
  0$ and hence $\rp d_n\to \rp e$, but this contradicts the observation that in a
  neighbourhood of $U$ every effective class has non-negative self-intersection.
  Therefore (\ref{thm:main}.1) is proven.

  Now let $\xi\in \NE(X)$ extremal. If no multiple of $\xi$ is in $\sN(X)$, then
  $\xi\in \sQ(X)$ by \eqref{cor:non-nodal-extremal}. If $\xi$ were not contained in
  the closure of the convex cone generated by $\sN(X)$, then by \cite[2.6]{Kovacs94}
  $\NE(X)$ would have a circular part, so (\ref{thm:main}.2) follows from
  (\ref{thm:main}.1).

  Finally, since every class in $\NE(X)$ may be written as a sum of finitely many
  extremal classes (\ref{thm:main}.3) follows from (\ref{thm:main}.2).
  \qed
\end{proof}

\begin{cor}\label{cor:short-main}
  Let $X$ be a K3 surface of Picard number at least three over an algebraically
  closed field of arbitrary characteristic. Then one of the following mutually
  exclusive conditions are satisfied:
  \begin{align}
  \NE(X) & =\conv(\sQ(X)),\text{ or}\\
  \NE(X) & =\overline{\sum_{\ell\in\sN(X)}\rp \ell}.
  \end{align}
\end{cor}

\begin{proof}
  If $X$ does not contain any curve of negative self-intersection, then
  (\ref{cor:short-main}.1) follows from (\ref{eq:1}).  Otherwise
  (\ref{cor:short-main}.2) follows from \eqref{thm:main}.
  \qed
\end{proof}

\begin{rem}
  Clearly \eqref{cor:short-main} is equivalent to \eqref{thm:intro}.
\end{rem}

%\bibliographystyle{/home/kovacs/tex/TeX_input/skalpha} %$ 
%\bibliography{/home/kovacs/tex/TeX_input/Ref} %$

\def\cprime{$'$} \def\polhk#1{\setbox0=\hbox{#1}{\ooalign{\hidewidth
  \lower1.5ex\hbox{`}\hidewidth\crcr\unhbox0}}} \def\cprime{$'$}
  \def\cprime{$'$} \def\cprime{$'$} \def\cprime{$'$}
  \def\polhk#1{\setbox0=\hbox{#1}{\ooalign{\hidewidth
  \lower1.5ex\hbox{`}\hidewidth\crcr\unhbox0}}} \def\cdprime{$''$}
  \def\cprime{$'$} \def\cprime{$'$} \def\cprime{$'$} \def\cprime{$'$}
\providecommand{\bysame}{\leavevmode\hbox to3em{\hrulefill}\thinspace}
\providecommand{\MR}{\relax\ifhmode\unskip\space\fi MR}
\providecommand{\MRhref}[2]{%
  \href{http://www.ams.org/mathscinet-getitem?mr=#1}{#2}
}
\providecommand{\href}[2]{#2}

\end{document}